\documentclass[twoside,a4paper,12pt]{amsart}

\input xypic
\usepackage{amssymb,enumitem}
\usepackage[T1]{fontenc}
\xyoption{all}

\newtheorem{thm}{Theorem}[section]
\newtheorem{lemma}[thm]{Lemma}

\newtheorem{prop}[thm]{Proposition}

\author{Jeremy Rickard} 
\address{School of Mathematics, University of Bristol, Bristol BS8 1TW, UK}
\email{j.rickard@bristol.ac.uk}
\date\today

\title{Pathological abelian groups: a friendly example}
\begin{document}
\maketitle
\begin{abstract}
  We show that the group of bounded sequences of elements of
  $\mathbb{Z}[\sqrt 2]$ is an example of an abelian group with several
  well known, and not so well known, pathological properties.  It
  appears to be simpler than all previously known examples for some of
  these properties, and at least simpler to describe for others.
\end{abstract}

\section{History and introduction}

In the first edition of his famous book {\em Infinite Abelian Groups}
in 1954~\cite{kaplansky}, Kaplansky proposed three ``test problems''
for abelian groups. The motivation was that if these did not have
positive answers for some particular class of groups, then we could
pretty much give up on a satisfactory structure theorem for that
class. At the time, all three problems were open for general abelian
groups, although they were all answered during the next decade.

The first problem asked whether, if two abelian groups $G$ and $H$
are each isomorphic to a direct summand of the other, then we must
have $G\cong H$. In 1961 S\k{a}siada~\cite{sasiada} provided a
counterexample.

The second asked whether, if $G$ and $H$ are abelian groups with
$G\oplus G$ and $H\oplus H$ isomorphic, then we must have $G\cong
H$. In 1957 J\'onsson~\cite{jonsson} provided a counterexample.
 
Subsequently, all manner of pathological phenomena concerning direct
sum decompositions of abelian groups have been discovered, in many
cases involving quite intricate constructions. One of the most
well known is Corner's discovery~\cite{corner} of an abelian group $G$
such that $G\cong G\oplus G\oplus G$, but $G\not\cong G\oplus G$. This
example gives another solution to Kaplansky's first two problems,
since if we take $H=G\oplus G$, then $G$ and $H$ provide
counterexamples for both problems.

In a later edition of his book, Kaplansky wrote ``In this strange part
of the subject anything that can conceivably happen actually does
happen.''

However, Kaplansky's third test problem asked whether, if
$G\oplus\mathbb{Z}$ is isomorphic to $H\oplus\mathbb{Z}$, then we must
have $G\cong H$. This was answered independently by Cohn~\cite{cohn}
and Walker~\cite{walker} in 1956, and this time the answer is ``yes''.

Perhaps this indicates that, although the bad behaviour of general
abelian groups can be almost limitless, they can sometimes be
restrained by the company of well-behaved groups like $\mathbb{Z}$.

On the internet site MathOverflow, Martin
Brandenburg~\cite{brandenburg} asked in 2015 whether there is an
abelian group $A$ such that $A\cong A\oplus\mathbb{Z}^2$, but
$A\not\cong A\oplus\mathbb{Z}$. In fact the same question had been
answered in 1985 by Eklof and Shelah~\cite{eklof_shelah}, with a
characteristically ingenious and intricate construction. They credit
Sabbagh with asking the question.

The purpose of this paper is to give a simpler example of such a
group, which also provides the simplest example that I know of
Corner's phenomenon (and hence provides examples that answer
Kaplansky's first two test problems, maybe not simpler to verify than
existing examples, but easier to describe).

The group in question is the group of bounded sequences of elements of
$\mathbb{Z}[\sqrt 2]$.

\section{The main theorem}
\label{sec:main}

We will frequently be dealing with sequences of numbers, or of
functions, and we will use the compact notation $\underline{x}$ for a
sequence $(x_0,x_1,\dots)$. By a finite sequence we will mean
one with only finitely many nonzero terms.

For this section and the next, $A$ will be the additive group of
bounded (in the real norm) sequences $\underline{a}$ of elements of
$\mathbb{Z}[\sqrt 2]$. Our main theorem is

\begin{thm}\label{thm:main}
  As abelian groups, $A\cong A\oplus\mathbb{Z}^2$, but
  $A\not\cong A\oplus\mathbb{Z}$.
\end{thm}

Clearly $A\oplus\mathbb{Z}[\sqrt 2]\cong A$ via the map
$(\underline{a},b)\mapsto(b,a_0,a_1,\dots)$. Since
$\mathbb{Z}[\sqrt 2]\cong\mathbb{Z}\oplus\mathbb{Z}$ as abelian
groups, $A\cong A\oplus\mathbb{Z}\oplus\mathbb{Z}$. We shall prove
that $A\not\cong A\oplus\mathbb{Z}$.

We first want to understand group homomorphisms
$\varphi:A\to\mathbb{Z}[\sqrt 2]$. We shall show in
Proposition~\ref{prop:functional} that there are none apart from the
obvious ones. This is analogous to a well known result of
Specker~\cite{specker} stating that the only group homomorphisms from
the group of all integer sequences (the Baer-Specker group) to
$\mathbb{Z}$ are the obvious ones of the form
$\underline{x}\mapsto\sum_ib_ix_i$ for a finite sequence
$\underline{b}$ of integers. We will adapt one proof of Specker's
result that combines ideas of S\k{a}siada and \L o\'s.

For $n\in\mathbb{N}$, let $A_n\cong\mathbb{Z}[\sqrt 2]$ be the
subgroup of $A$ consisting of sequences whose terms are all zero,
apart from possibly the $n$th term.

\begin{lemma}\label{lem:finite}
  Let $\varphi:A\to\mathbb{Z}[\sqrt 2]$ be a group homomorphism. Then
  $\varphi(A_n)=0$ for all but finitely many $n$.
\end{lemma}

\begin{proof} (cf. S\k{a}siada~\cite{sasiada:slender}.)
  If not, we can choose $n_0<n_1<\dots$ so that
  $\varphi(A_{n_k})\neq0$ for all $k$. The intersection of $A_{n_k}$
  with $\ker(\varphi)$ has rank at most one, so we can inductively
  choose $x_k\in A_{n_k}$ so that $\varphi(x_k)\neq0$,
  $\vert x_k\vert<1$, and $x_k$ is divisible by a larger power of $2$
  than any of $\varphi(x_0),\dots,\varphi(x_{k-1})$.

  Consider the sequences in $A$ whose $n_k$th term, for each $k$, is
  either $x_k$ or $0$, with all other terms zero. Since there are
  uncountably many such sequences, $\varphi$ must agree on two of
  them. Taking the difference of these two, we get a nonzero sequence
  in $\ker(\varphi)$ whose first nonzero term, in the $n_k$th place
  for some $k$, is $\pm x_k$ and with all other terms divisible by a
  higher power of $2$ than $\varphi(x_k)$. But this is a
  contradiction, since $\varphi(x_k)=\pm\varphi(y)$, where $y$ is the
  sequence obtained by replacing the first non-zero term by zero.
\end{proof}

Let $A'<A$ be the subgroup
$\left\{\underline{a}\in A\mid \sum_na_n\mbox{ converges}\right\}$. We
can adapt the previous proof, by replacing the condition
$\vert x_k\vert<1$ with the condition $\vert x_k\vert<2^{-k}$, to get
the following variant of the lemma.

\begin{lemma}\label{lem:variant}
  Let $\varphi:A'\to\mathbb{Z}[\sqrt 2]$ be a group homomorphism. Then
  $\varphi(A_n)=0$ for all but finitely many $n$.
\end{lemma}

\begin{lemma}\label{lem:zero}
  Let $\varphi:A\to\mathbb{Z}[\sqrt 2]$ be a group homomorphism such that
  $\varphi(A_n)=0$ for all $n$. Then $\varphi=0$.
\end{lemma}

\begin{proof} (cf. \L o\'s~\cite{los}.)  Suppose not. Choose
  $\underline{a}\in A$ with $\varphi(\underline{a})\neq0$. If
  $\underline{b}\in A'$ then the partial sums $b_0+b_1+\dots+b_n$ are
  bounded, so we can define a homomorphism $\theta:A'\to A$ by
$$\theta(\underline{b})=\left(b_0a_0,(b_0+b_1)a_1,(b_0+b_1+b_2)a_2,\dots\right).$$

If $\underline{e_k}\in A'$ is the sequence which is zero except that
the $k$th term is $1$, then $\theta(\underline{e_k})$ differs from
$\underline{a}$ only in the first $k$ terms, so
$\varphi\theta\left(\underline{e_k}\right)=\varphi(\underline{a})\neq0$ for every
$k$, contradicting Lemma~\ref{lem:variant}.
\end{proof}

\begin{prop}\label{prop:functional}
  The only group homomorphisms $\varphi:A\to\mathbb{Z}[\sqrt 2]$ are
  the maps
$$\underline{a}\mapsto\sum_n\varphi_n(a_n),$$
for finite sequences $\underline{\varphi}$ of group
homomorphisms $\varphi_n:\mathbb{Z}[\sqrt 2]\to\mathbb{Z}[\sqrt 2]$.
\end{prop}

\begin{proof}
  Clearly every such $\varphi$ is a group homomorphism. By
  Lemmas~\ref{lem:finite} and \ref{lem:zero} every group homomorphism
  is of this form.
\end{proof}

\begin{prop}\label{prop:groupendo}
  Every group endomorphism $\beta:A\to A$ is determined by the
  compositions $\beta_{mn}:A_n\to A\stackrel{\beta}{\to}A\to A_m$,
  where, for each $m$, all but finitely many $\beta_{mn}$ are zero.
\end{prop}

\begin{proof}
  Since $A$ is a subgroup of a direct product of copies of
  $\mathbb{Z}[\sqrt 2]$ in an obvious way, this follows immediately
  from Proposition~\ref{prop:functional}.
\end{proof}

In other words, this means that if we think of sequences as infinite
column vectors then we can represent $\beta$ as an infinite matrix of
homomorphisms $\beta_{mn}:\mathbb{Z}[\sqrt 2]\to\mathbb{Z}[\sqrt 2]$,
with finitely many nonzero entries in each row. In fact, every such
matrix will describe a homomorphism from $A$ to the group of all
sequences of elements of $\mathbb{Z}[\sqrt 2]$, but extra conditions
are needed to ensure that the image of this homomorphism consists of
bounded sequences.

\begin{lemma}\label{lem:epsilon}
  Let $\vartheta:\mathbb{Z}[\sqrt 2]\to\mathbb{Z}[\sqrt 2]$ be a group
  homomorphism that is not a $\mathbb{Z}[\sqrt 2]$-module
  homomorphism. Then for any $\epsilon>0$ and $N>0$ there is some
  $x\in\mathbb{Z}[\sqrt 2]$ with $\vert x\vert<\epsilon$ and
  $\vert\vartheta(x)\vert>N$.
\end{lemma}

\begin{proof}
  $\vartheta(\sqrt 2)-\sqrt 2\vartheta(1)\neq0$, since $\vartheta$ is
  not a $\mathbb{Z}[\sqrt 2]$-module homomorphism. Choose sequences
  $\underline{a}$ and $\underline{b}$ of nonzero integers such that
  $a_n+b_n\sqrt 2\to 0$, so that necessarily $|b_n|\to\infty$. Then
$$\left|\vartheta(a_n+b_n\sqrt 2)\right|=
\left|(a_n+b_n\sqrt 2)\vartheta(1)+
b_n\left(\vartheta(\sqrt 2)-\sqrt 2\vartheta(1)\right)\right|
\to\infty.$$
\end{proof}

\begin{lemma}\label{lem:modulehom}
  Let $\beta:A\to A$ be a group endomorphism, and define
  $\beta_{mn}:A_n\to A_m$ as in Proposition~\ref{prop:groupendo}. For
  all but finitely many $n$, $\beta_{mn}$ is a
  $\mathbb{Z}[\sqrt 2]$-module homomorphism for all $m$.
\end{lemma}

\begin{proof}
  Suppose not. Since, for each $m$, $\beta_{mn}=0$ for all but
  finitely many $n$, we can recursively choose $m_0<m_1<\dots$ and
  $n_0<n_1<\dots$ so that $\beta_{m_kn_k}$ is not a
  $\mathbb{Z}[\sqrt 2]$-module homomorphism for any $k$, and such that
  $\beta_{m_kn_l}=0$ for $k<l$. For if we have chosen
  $m_0,\dots,m_{k-1}$ and $n_0,\dots,n_{k-1}$ consistent with these
  properties, then we can choose $n_k$ large enough that
  $\beta_{mn_k}=0$ for all $m\leq m_{k-1}$ and such that there is some
  $m_k$ for which $\beta_{m_kn_k}$ is not a
  $\mathbb{Z}[\sqrt 2]$-module homomorphism.

  By Lemma~\ref{lem:epsilon} we can recursively choose
  $x_{n_k}\in\mathbb{Z}[\sqrt 2]$ such that, identifying each $A_n$
  with $\mathbb{Z}[\sqrt 2]$ in the obvious way, $|x_{n_k}|$ is
  bounded, but
$$\left|\beta_{m_kn_k}(x_{n_k})+\sum_{l<k}\beta_{m_kn_l}(x_{n_l})\right|\to\infty.$$

But then, if we set $x_n=0$ for $n\not\in\{n_0,n_1,\dots\}$, the
sequence $\underline{x}$ is bounded, and so in $A$, but
$\beta(\underline{x})$ is unbounded, contradicting the fact that
$\beta(\underline{x})\in A$.
\end{proof}

\begin{proof}[Proof of Theorem~\ref{thm:main}]
  As pointed out earlier, the fact that $A\cong A\oplus\mathbb{Z}^2$
  is clear.

 Suppose $A\cong A\oplus\mathbb{Z}$. Then there is a monomorphism
 $\beta:A\to A$ such that $A/\beta(A)\cong\mathbb{Z}$. By
 Proposition~\ref{prop:groupendo} and Lemma~\ref{prop:groupendo},
 $\beta$ is described by an infinite matrix $(\beta_{mn})$ with only
 finitely many columns containing any entries that are not
 $\mathbb{Z}[\sqrt 2]$-module homomorphism. So for sufficiently large
 $n$, if $A[n]\leq A$ is the subgroup of sequences whose first $n$
 terms are zero, then the restriction of $\beta$ to $A[n]$ is a
 $\mathbb{Z}[\sqrt 2]$-module homomorphism, and so $A/\beta(A[n])$ is
 a $\mathbb{Z}[\sqrt 2]$-module, and so must have even rank, since
 $\left(A/\beta(A[n])\right)\otimes_\mathbb{Z}\mathbb{Q}$ is a vector
 space over $\mathbb{Q}(\sqrt 2)$. However, as an abelian group
 $A/\beta(A[n])\cong\mathbb{Z}^{2n+1}$.
\end{proof}

\section{Corollaries}
\label{sec:corollaries}

For $A$ the group considered in Section~\ref{sec:main}, let
$B=A\oplus\mathbb{Z}$, so $A\not\cong B$ by
Theorem~\ref{thm:main}. Using the fact that $A\cong A\oplus A$, it is
easy to see that $B$ provides an example of Corner's phenomenon.

\begin{thm}\label{thm:corner}
  $B\cong B\oplus B\oplus B$, but $B\not\cong B\oplus B$.
\end{thm}

\begin{proof}
  The map
  $(\underline{a},\underline{b})\mapsto (a_0,b_0,a_1,b_1,\dots)$ is an
  isomorphism of abelian groups $A\oplus A\to A$, so
$$B\oplus B\cong A\oplus\mathbb{Z}\oplus A\oplus\mathbb{Z}\cong A\oplus A
\cong A\not\cong B.$$

However,
$$B\oplus B\oplus B\cong B\oplus A\cong A\oplus\mathbb{Z}\oplus A
\cong A\oplus\mathbb{Z}=B.$$
\end{proof}

Consequently, $A$ and $B$ also answer Kaplansky's first two test
problems: each is a direct summand of the other, and
$A\oplus A\cong B\oplus B$.

\section{Variants}
\label{sec:variants}

Now let $n>1$ and let $A$ be the group of bounded sequences of
elements of $\mathbb{Z}[\sqrt[n] 2]$. 

A straightforward adaptation of the proof of our main theorem shows
that $A\cong A\oplus\mathbb{Z}^n$ but $A\not\cong A\oplus\mathbb{Z}^m$
for $0<m<n$. Eklof and Shelah~\cite{eklof_shelah} also give such an
example.

Also, Theorem~\ref{thm:corner} generalizes to show that
$B=A\oplus\mathbb{Z}$ is isomorphic to the direct sum of $n+1$ copies
of itself, but not to the direct sum of $m+1$ copies if $0<m<n$.

\bibliography{mybib}{} 
\bibliographystyle{amsalpha}

\end{document}